\documentclass[9pt]{extarticle}
\usepackage{amsmath, amssymb, amsthm, relsize, extsizes}

\theoremstyle{plain}
\newtheorem{theorem}{Theorem}[section]

\theoremstyle{definition}
\newtheorem{definition}[theorem]{Definition}

\DeclareMathOperator*{\OmSum}{\mathlarger{\mathlarger{\Omega}}}

\begin{document}
\title{The Compositional Integral\\A Brief Introduction}
\author{James David Nixon\\
	JmsNxn92@gmail.com\\
	University of Toronto}

\maketitle

\begin{abstract}
The Compositional Integral is defined, formally constructed, and discussed. A direct generalization of Riemann's construction of the integral; it is intended as an alternative way of looking at First Order Differential Equations. This brief notice aims to: familiarize the reader with a different approach to integration, fabricate a notation for a modified integral, and express a startling use for infinitely nested compositions. Taking inspiration from Euler's Method for approximating First Order Differential Equations, we affiliate the method with Riemann Sums; and look at it from a different, modern angle.
\end{abstract}



\section{Introduction}\label{sec1}
\setcounter{equation}{0}

As an undergraduate student in mathematics we’ve all encountered (or will encounter) a theorem by Picard and Lindel\"{o}f. Though this theorem is named for \'{E}mile Picard and Ernst Lindel\"{o}f, its history traces from Augustin-Louis Cauchy, to Leonhard Euler, to Isaac Newton--and through much of 17th to 19th century mathematics. But as ingenious Picard and Lindel\"{o}f’s theorem is, it is non-productive; insofar as it does not \emph{produce} the function in any feasible manner. For a more detailed look at the history and development of differential equations, refer to \cite{HisODEMisc, HisODECaj, HisODEInce, HisODESas}; where arguably the climax of classical contributions to the theory of differential equations is Picard and Lindel\"{o}f's eponymous theorem. This paper aims to look at First Order Differential Equations from a different perspective; like the Necker cube, differential equations can be viewed in more ways than one.\\

The brute can summarize Picard and Lindel\"{o}f's theorem in a few key words. For $|x-x_0| < \delta$ with $\delta > 0$ small enough, and for $f(x, t)$ a Lipschitz continuous function, the mapping:

\[
\Phi u = y_0 + \int_{x_0}^x f(s,u(s))\,ds
\]

is a contraction. Therefore, by The Banach Fixed Point Theorem\footnote{This theorem will most likely be the first non-trivial case where a student encounters a use for The Banach Fixed Point Theorem.} it has a unique fixed point $y$, which inherently satisfies $y(x_0) = y_0$ and $\Phi y = y$, which reduces to $y'(x) = f(x, y(x))$.\\

This provides us with a local solution to the differential equation $y' = f(x,y)$ subject to the constraint $y(x_0) = y_0$. It also ensures this solution is unique, thanks to our use of The Banach Fixed Point Theorem; and that it is continuously differentiable. A more in depth look at The Picard-Lindel\"{o}f Theorem can be found in \cite{TheOdDEq}. 

The author would like to expand this theorem. This theorem does not necessarily help us produce the solution $y$–it only informs us it exists and is unique.  Even though Picard and Lindel\"{o}f's theorem is considered constructive--the author feels it isn't productive. The result is restricted to tiny intervals about $x_0$, making us unsure of where the iteration of $\Phi$ actually converges--simply that it must somewhere. Numerical calculations of $y$ using iterations of $\Phi$ is also unfeasible in practice.

In no different a manner than how the Riemann Integral is \emph{productive}, we should have something similar for $y$. There is a Riemann Sum of $g$ which converges to an object $G$ and that object satisfies $G'(x) = g(x)$. What if there was the same thing for the equation $y' = f(x, y(x))$? A kind of ``Picard-Lindel\"{o}f Integral." We should have some \emph{thing} and this \emph{thing} converges to $y$ everywhere; and in a useful productive sense. 

This \emph{thing} should also behave in a manner similar to our usual notion of an integral. This \emph{thing} should be accessible and intuitive like the Riemann Sum. And above all, this thing should look and \emph{feel} like an integral. The author will argue that Euler had already met these criteria in the 18th century. He simply worded it in a language of infinitesimals.\\

We are going to start with a formal calculus and prove some less than obvious facts about it. The author will then propose a way of actually constructing this formal calculus. In order to do this, the author will be brash and introduce a new notation known as the differential bullet product. This will take some persuasion. He maintains though, that the summation of this paper is the introduction of a formal calculus, The Compositional Integral, and an argument for a notation which describes said idea.

The methodology of our proposed, new kind of integral, is vastly similar to Euler's Method for approximating First Order Differential Equations. The idea is to make this approximation notion more precise, and re-approach it as a modified Riemann Sum--and describe some properties which, as the author would put it, have been overlooked. We are also going to steal Leibniz's notation for the integral, and twist it a bit.

However, and this is a strong however, the integral was originally a formal idea. It took much work to make it anything less than formal. A Riemann Sum, as beautiful and practical as it is, still couldn’t accomplish a victory over the formal meaning. The integral is a mysterious thing. The author cannot fully construct The Compositional Integral to the extent mathematician’s have constructed modern integration. He simply wishes to discuss the formal object, and give a concrete instance where it works. Therein, the majority of this paper will be formal arguments. He hopes more knowledgeable mathematicians of measure theory and Lesbesgue's theory of integration will have something more interesting to say.

\section{So, what is The Compositional Integral?} \label{sec2}
\setcounter{equation}{0}

Let’s borrow Leibniz's notation for the integral. The Compositional Integral can be introduced modestly in a less than avant-garde fashion. The notation may look a little clunky, but the pieces fit together rather tightly.

Let $b \ge a$, supposing $f(s,t)$ is a \emph{nice} function\footnote{Bear with the author, as what we mean by \emph{nice} will have to be filled in as we progress.}, write:

\begin{equation}\label{eq:1A}
Y_{ba}(t) = \int_a^b f(s, t)\, ds \bullet t
\end{equation}

To get what this expression means will be the point of this paper. And if the reader can absorb what this expression means, they can absorb the thesis of this paper. The author’s goal is to acclimatize the reader slowly with this notation. But the author will simply start with the denotion $Y_{ba}(t)$.\\

\begin{definition}[The Compositional Integral]\label{def1}
The Compositional Integral:

\[
y(x) = Y_{xa}(t) = \int_a^x f(s, t)\, ds \bullet t
\]

is the \emph{unique}\footnote{Although one would usually have to show $y$ is unique, by The Picard-Lindel\"{o}f Theorem it certainly is if, for instance, $f$ is globally Lipschitz on its domain. We include uniqueness in the definition for convenience.} function $y$ such that $y'(x) = f(x, y(x))$ and $y(a) = t$.
\end{definition}

This is a bit of a mouthful, and imprecise on domains, but the imprecision of this definition is warranted. This definition is made the way it is to introduce more simply what the author calls the formal semi-group laws; where $a \le b \le c$:

\begin{enumerate}
	\item $Y_{aa}(t) = t$\\
	\item $Y_{cb}(Y_{ba}(t)) = Y_{ca}(t)$
\end{enumerate}

These laws comprise a modified additivity condition of the usual integral $\int_b^c + \int_a^b = \int_a^c$ and $\int_a^a = 0$. Where now addition is replaced with composition--and we have a semi-group-structure rather than a group-structure (at least for now). 

As a brief digression, to gather some intuition; if we were to let $f(s,t)= f(s)$ be constant in $t$, then the differential equation would reduce to $y'(x) = f(x)$ and $y(a) = t$. The Compositional Integral becomes the integral. That is to mean $y(x) = t + \int_a^xf(s)\,ds$, and $Y_{ba}(t) = t + \int_a^b f(s)\,ds$.  The composition law ($2$) across $t$ becomes the usual additivity condition of the integral--albeit written a bit strangely. Of which, the constant of integration plays a more prominent role as an argument of a function.

There is not much more than a trick to proving ($1$) and ($2$). We will restrict ourselves to a formal proof that breaks down the mechanism of it purely from Definition \ref{def1}. The following argument really only works for well behaved $f$, which we avoid describing as it may muddy the initial intuition. But the reader may guess that it is something like a \emph{nice} global Lipschitz condition.

\begin{theorem}\label{thm1} 
Let $Y_{ba}(t)$ be The Compositional Integral of $f(s,t)$; then the following group laws are satisfied:

\begin{enumerate}
	\item $Y_{aa}(t) = t$\\
	\item $Y_{cb}(Y_{ba}(t)) = Y_{ca}(t)$\\
\end{enumerate}

for all $a \le b \le c$.
\end{theorem}

\begin{proof}
Using the definition of $Y$ we can play a few tricks to get our result. When $y(x)= Y_{xa}(t)$ then $y(a) = t$ by definition and so ($1$) is satisfied by first principles. Proving ($2$) is more wordplay then anything. Firstly $u = Y_{xb}(Y_{ba}(t))$ is the unique function such that $u(b) = Y_{bb}(Y_{ba}(t))= Y_{ba}(t)$ (by ($1$)) and $u'(x) = f(x, u(x))$. Similarly though, the function $w = Y_{xa}(t)$ is the unique function such that $w(b) = Y_{ba}(t)$ and $w'(x) = f(x, w(x))$. Therefore they must equal, $w = u$. Plugging in $x = c$ gives the result.
\end{proof}

The majority of this theorem relied on the uniqueness of a solution to a First Order Differential Equation; where again $f$ is \emph{nice}. This allowed for an identity principle, which was used as the cornerstone of this theorem. It isn't very hard to imagine the cases where $f$ is suitable and this argument works--again, something like a \emph{nice} global Lipschitz condition.\\

Now, this identity alone does not justify considering this object an integral. Luckily, there's more hidden to the proposed notation. Considering the usual integral, there is an iconic ability to substitute variables using Leibniz's differential calculus. If $u = \gamma(s)$, $du = \gamma'(s)ds$ and $u(\alpha) = a$ and $u(\beta) = b$ then,

\[
\int_a^b f(s)\,ds = \int_\alpha^\beta f(u)\, du = \int_\alpha^\beta f(\gamma(s))\gamma'(s)\,ds
\]

This leads us to the next nice fact about The Compositional Integral, and hints more aggressively as to the usefulness of the proposed notation. The same substitution of variables is still perfectly valid.

\[
\int_a^b f(s,t) \,ds\bullet t = \int_{\alpha}^\beta f(u,t) \,du \bullet t = \int_\alpha^\beta f(\gamma(s),t)\gamma'(s)\,ds\bullet t
\]

Remembering the definition of The Compositional Integral, this can be shown using the less than startling identity:

\[
\frac{d}{dx} y(\gamma(x)) = f(\gamma(x),y(\gamma(x)))\gamma'(x)
\]

More thoroughly, the function $w(x) = y(\gamma(x))$ is the \emph{unique} function such that $w'(x) = f(\gamma(x),w(x))\gamma'(x)$ and $w(\alpha) = y(\gamma(\alpha)) = y(a) = t$. Therefore $w(x) = \int_\alpha^x f(\gamma(s),t)\gamma'(s)\,ds\bullet t$. Similarly $w(\beta) = y(b)$. Since $y(b) = Y_{ba}(t)$, we must have:

\[
Y_{ba}(t) = \int_\alpha^\beta f(\gamma(s),t)\gamma'(s)\,ds\bullet t
\]

This also explicitly constructs the inverse for each element of the semi-group. Meaning, The Compositional Integral forms a group under composition. The function $Y_{ab}(t) = Y_{ba}^{-1}(t)$ which allows $Y_{ab}(Y_{ba}) = Y_{aa} = t$. Theorem \ref{thm1} hypothetically shows this, but it may be unclear what $Y_{ab}$ means when $a \le b$. Using Leibniz's rule of substitution where $u(s) = b+a - s$ with $u(a) = b$ and $u(b) = a$, the meaning of $Y_{ab}$ can be clarified by the identity:

\[
Y_{ab}(t) = \int_b^a f(s,t)\,ds \bullet t = \int_a^b f(u,t)\,du \bullet t = \int_a^b -f(b+a-s,t)\,ds \bullet t
\]

Therefore:

\[
Y_{ab}(t) = Y^{-1}_{ba}(t) = \int_a^b -f(b+a-s,t)\,ds \bullet t
\]

This leaves us with the conception that not only does $Y_{ba}(t)$ have a semi-group-structure for $a \le b$, it has a group-structure when we remove the restriction $a \le b$; the inverse of $Y_{ba}$ is $Y_{ab}$, which can be described using a substitution of variables.\\

The last facet of The Compositional Integral is perhaps the most important. The Compositional Integral can be constructed using something looking like a Riemann Sum. We will devote the majority of this brief paper justifying this.

The group structure of $Y_{ba}$ can be used to construct $Y_{ba}$. It is beneficial to think accurately about Euler's Method in the following argument. If throwing in one's hat as to whom truly deserves priority over all the ideas in this paper; the author feels the entirety of this paper could probably be attributed to Euler, it's simply that he said it differently.

Let $P = \{s_i\}_{i=0}^n$ be a partition of $[a, b]$ written in descending order (this can cause a bit of a trip up). That is to say $b = s_0 > s_1 > s_2 > ... > s_n = a$. Let's also write $Y_{\Delta s_i} = Y_{s_is_{i+1}}$ then by the group law $Y_{ca} = Y_{cb}(Y_{ba})$:

\[
Y_{ba} = Y_{bs_1}(Y_{s_1s_2}(Y_{s_2s_3}(...Y_{s_{n-1}a}))) = \prod_i Y_{\Delta s_i} 
\]

Where the product is taken to mean composition. As $\Delta s_i = s_i -s_{i+1}$ tends to zero these $Y_{\Delta s_i} \to t$, which follows because $Y_{aa}(t) = t$. We know something stronger though, we know as $s_i,s_{i+1} \to s_{i}^*$ that $\dfrac{Y_{\Delta s_i} - t}{\Delta s_i} \to f(s_i^*, t)$ which is just the differential equation used to define $Y$. This implies each $Y_{\Delta s_i}$ looks like $t + f(s_i^*, t)\Delta s_i$, up to an error of order $\mathcal{O}(\Delta s_i^2)$. To make an educated guess then: shouldn't the $\mathcal{O}(\Delta s_i^2)$ part be negligible as the partition gets finer? It should be safe to write:

\[
Y_{ba} = \lim_{\Delta s_i \to 0} \prod_i t + f(s_i^*, t)\Delta s_i
\]

We may write this in a form that would be familiar to more Classical Analysts. Let $\Delta s_i \to ds$ and write:

\begin{equation}
Y_{ba} = \prod t + f(s, t)\, ds\\
\label{eq:1}
\end{equation}

Although this may seem unusual; this is no more than the combination of Euler's Method and Riemann Sums with their language of partitions. For context, Euler would usually fix $\Delta s_i$ small and approximate $Y_{ba}$ as we let $b$ grow. We are fixing $b$, but letting $\Delta s_i \to ds$ and thinking of this as a Riemann Sum. But then again, Classical Analysts never saw the need for a Riemann Sum, it was just how infinitesimals worked.

\section{But where does the $\bullet$ come from?}\label{sec3}
\setcounter{equation}{0} 

If the reader leaves this paper with one thing, it is hopefully an understanding of the differential bullet product $ds \bullet t$. And what it means when combined with the integral. The expression $\int ds \bullet t$ behaves similarly to the expression $\int ds$ (they both satisfy: a group structure, substitution of variables, and a First Order Differential Equation).

But, what does the differential bullet product mean? That’s a tough question to answer without sufficient context. The author will boil it down into one thing. We cannot use the notation $\prod$ to represent nested compositions. Notation must represent clearly and precisely. We’ll need a notation for iterated compositions, and we’ll have to be clear. The author chooses the symbol $\OmSum$.

If $h_j$ is a sequence of functions taking some interval to itself. The expression $\OmSum_{j=0}^n h_j (t)$ can be understood to mean:

\[
\OmSum_{j=0}^n h_j (t) = h_0(h_1(h_2(...h_n(t))))
\]

No different than Euler’s notation for products and sums, except, we’ll need some additional notation in the spirit of Leibniz. When our function $h_j$ depends on another variable which is not being composed across, our notation becomes unclear. Insofar, when we write:

\[
\OmSum_{j=0}^n h_j (s, t)
\]

Does this mean we compose across $t$?

\[
h_0(s, h_1(s, ...h_n(s, t)))
\]

Or does this mean we compose across $s$?

\[
h_0(h_1(...h_n(s, t)..., t), t)
\]

This is no different a problem than when an undergraduate writes $\int e^{st}$ and the professor is expected to guess whether the integration is across $s$ or $t$. It becomes unclear whether one means $\int e^{st}\,ds$ or $\int e^{st} \, dt$. To reconcile the situation we’re going to use a bullet. Therein, the above expressions can be written more clearly:

\[
\OmSum_{j=0}^n h_j (s, t) \bullet t = h_0(s, h_1(s, ...h_n(s, t)))
\]

\[
\OmSum_{j=0}^n h_j (s, t) \bullet s = h_0(h_1(...h_n(s, t)..., t), t)
\]

So when the author uses a bullet ($\bullet$) followed by a variable, it is to mean that the operation is bound to the variable. Specifically compositions are across this variable. 

Returning to our discussion from above, we arrive at a more playful denotion of The Compositional Integral. Let $P = \{s_i\}_{i=0}^n$ be a partition of $[a, b]$ (written in descending order) with $s_{i+1} \le s_i^* \le s_i$, then we can write, and aim to justify:

\[
\prod_i t + f(s_i^*,t)\Delta s_i = \OmSum_{i=0}^{n-1} t + f(s_i^*, t)\Delta s_i \bullet t \approx Y_{ba}(t)
\]

This is essentially the statement of Euler's Method as it's stated today, and it's still used by numerical approximation algorithms. Traditionally one would write this calculation a bit more sequentially:

\begin{eqnarray*}
t_0 &=& t\\
t_1 &=& t_0 + f(s_{n-1}^*,t_0)\Delta s_{n-1}\\
t_2 &=& t_1 + f(s_{n-2}^*,t_1)\Delta s_{n-2}\\
&\vdots&\\
t_n &=& t_{n-1} + f(s_0^*,t_{n-1})\Delta s_0 \approx Y_{ba}(t)\\
\end{eqnarray*}

The benefit of this notation is that $t_{n+1} \approx Y_{(b+\Delta)a}$ and $t_{n+2} \approx Y_{(b+2\Delta)a}$, and we can think of this as a sequence, or a process, which continues to approximate. The main proposition of our altered form, is this becomes equality as $\Delta s_i \to 0$ (and $n\to\infty$), but $b$ is fixed, and isn't allowed to vary. This isn't much of a leap of faith considering the vast amount of numerical evidence which supports this claim; and lays at the foundation of numerical approximation algorithms. 

At this point, the notation can be rephrased; the notation of Section 2 can be better motivated. Let $\Delta s_i \to ds$, the summatory part $\OmSum t + $ becomes an $\int$. This is a continuous, infinitesimal, composition; similar to a continuous sum. It becomes a sweep of $f(s,t)$ for $s \in [a,b]$ across $t$; which we write as $ds \bullet t$. Again, this is something like a Riemann Sum... but it's an infinitely nested composition. The bounds on the integral can be made explicit, and it leaves us with the expression:

\[
\lim_{\Delta s_i \to ds}\OmSum_{i=0}^{n-1} t + f(s_i^*, t)\Delta s_i \bullet t = \int_a^b f(s, t)\, ds \bullet t
\]

Which is what the author means by the differential bullet product. We specifically call it a product as our group law $Y_{cb}(Y_{ba}(t)) = Y_{ca}$ can be written as the product of integrals:

\[
\int_{a}^c f(s,t)\,ds\bullet t = \int_{b}^c f(s,t)\,ds\bullet \int_{a}^b f(s,t)\,ds\bullet t
\]

The bullet is composition. We choose a bullet for this product of integrals, rather than the traditional symbol $\circ$, as to specify the composition is across $t$; and to emphasize the group-structure.\\

It is helpful to note when $f(s, t) = f(s)$ is constant in $t$, then $t + f(s)$ is a translation and the compositions above revert to addition. Illustrated by the formal manipulations,

\begin{eqnarray*}
\int_a^b f(s) ds \bullet t &=& \lim_{\Delta s_i\to0}\OmSum_{i=0}^{n-1} t + f(s_i^*)\Delta s_i \bullet t\\
&=& \lim_{\Delta s_i \to 0} \big{(}t + \sum_{i=0}^{n-1} f(s_i^*)\Delta s_i\big{)}\\
&=& t + \int_a^b f(s)\, ds\\
\end{eqnarray*}

We are reduced to the usual Riemann Sum definition of an integral. Furthermore we can explicitly see $t$ as the constant of integration. The additivity of integrals is again the composition law 

\[
Y_{cb}(Y_{ba}(t)) = t + \int_b^c f(s)\,ds + \int_a^b f(s)\,ds = t + \int_a^c f(s)\,ds = Y_{ca}(t)
\]

\section{Approaching from the other side of the equation}\label{sec4}
\setcounter{equation}{0}

Continuing with the same idea, we are going to approach from the other side of the equation. We will start with our (or Euler's?) proposed definition of $Y$ and argue that it is $Y$. To separate the objects as two different things we will call the proposed definition $\widetilde{Y}$. The aim is to formally argue and motivate $Y = \widetilde{Y}$, but we will not attempt to prove it yet. The benefit of this side, is to construct $\widetilde{Y}$ first, and provide a constructive/productive form of Picard and Lindel\"{o}f's Theorem.

Recalling the proposed definition: if $P = \{s_i\}_{i=0}^n$ is a partition of $[a, b]$ in descending order, and $s_{i+1} \le s_i^* \le s_i$:

\[
\widetilde{Y}_{ba}(t) = \lim_{\Delta s_i\to 0} \OmSum_{i=0}^{n-1} t + f(s_i^*, t) \Delta s_i \bullet t
\]

We are going to take a leap of faith momentarily and assume this expression converges uniformly. Some things about $\widetilde{Y}$ are simple to prove off hand, but $\widetilde{Y}$ may seem so foreign to navigate, the reader may not know where to look. So to start slow, the semi-group laws from before hold. Firstly, $\widetilde{Y}_{aa}(t) = t$ as this becomes the null composition which is the identity value $\text{Id} = t$. It is helpful to think about how the null sum is $0$, and the null product is $1$. More importantly, $\widetilde{Y}_{cb}(\widetilde{Y}_{ba}) = \widetilde{Y}_{ca}$ which is worth while to the reader for the author to write out.\\

\emph{A proof-sketch that $\widetilde{Y}_{cb}(\widetilde{Y}_{ba}) = \widetilde{Y}_{ca}$:} Let $P = \{s_i\}_{i=0}^n$ be a partition of $[a, b]$ in descending order, with $s_{i+1} \le s_i^* \le s_i$; and let $R = \{r_j\}_{j=0}^m$ be a partition of $[b, c]$ in descending order, with $r_{j+1} \le r_j^* \le r_j$. Then,

\begin{eqnarray*}
\widetilde{Y}_{cb}(\widetilde{Y}_{ba}) &=& \lim_{\Delta r_j\to 0} \OmSum_{j=0}^{m-1} t + f(r_j^*, t) \Delta r_j \bullet \lim_{\Delta s_i\to 0}\OmSum_{i=0}^{n-1} t + f(s_i^*, t)\Delta s_i \bullet t\\
&=& \lim_{\Delta r_j\to 0}  \lim_{\Delta s_i\to 0} \OmSum_{j=0}^{m-1} t + f(r_j^*, t) \Delta r_j \bullet\OmSum_{i=0}^{n-1} t + f(s_i^*, t)\Delta s_i \bullet t\\
&=& \lim_{\Delta q_k\to 0}\OmSum_{k=0}^{n+m-1} t + f(q_k^*, t)\Delta q_k \bullet t\\
&=& \widetilde{Y}_{ca}\\
\end{eqnarray*}

Where $Q = P \cup R= \{q_k\}_{k=0}^{n+m}$ is a partition of $[a, c]$ written in descending order, consisting of $q_j = r_j$ for $0 \le j \le m$ and $q_{i+m} = s_i$ for $0 \le i \le n$ and similarly for $q_k^*$.\\

This is purely a formal manipulation, but the reader may care to see how a rigorous proof would evolve if these objects converge uniformly in $t$, or in some \emph{nice} way. 

Reparametrizing the composition from $[a,b]$ to $[\alpha,\beta]$ with a continuously differentiable function $\gamma$; Leibniz's substitution of variables appears. Taking $\gamma(p^*_i) = s_i^*$ and $\gamma_i = \gamma_i(p_i) = s_i$, where $\beta = p_0 > p_1 > ... > p_{n-1} > p_n = \alpha$ and $p_{i+1} \le p_i^* \le p_i$, then by an approximate mean value theorem:

\[
f(s_i^*,t) (s_i - s_{i+1}) = f(\gamma(p_i^*),t) (\gamma_i - \gamma_{i+1}) \approx f(\gamma(p_i^*),t) \gamma'(p_i^*) (p_i - p_{i+1})
\]

So that,

\[
\widetilde{Y}_{ba}(t) = \lim_{\Delta p_i \to 0} \OmSum_{i=0}^{n-1} t + f(\gamma(p_i^*),t) \gamma'(p_i^*) \Delta p_i \bullet t
\]

So, our proposed definition also admits substitution of variables. It's also nice to see that the composition behaves little differently than how Riemann Sums behave, in this instance at least.\footnote{This hints aggressively to the idea of adding measure theory to the discussion.}

To extend our group-structure, if we invert $\widetilde{Y}_{ba}$ to $\widetilde{Y}_{ba}^{-1}$, then componentwise $t + f(s_i^*,t)\Delta s_i$ gets mapped to $\approx t - f(s_i^*,t)\Delta s_i$. Since composition is non-commutative, the partition is now in ascending order, and our inverse becomes precisely $\OmSum_{i=0}^{n-1} t - f(b+a-s_i^*,t)\Delta s_i \bullet t$. This agrees with our earlier inversion formula.\\

A more difficult idea to intuit is that $\widetilde{Y}_{xa}$, using this definition, satisfies the differential equation that $Y$ satisfies: $\frac{d}{dx} \widetilde{Y}_{xa} = f(x, \widetilde{Y}_{xa})$. And that this expression actually satisfies a First Order Differential Equation and we can come full circle. The author will only use intuition to morally justify this statement, for the moment. This logical sequence is a formal use of infinitesimals. Starting with the following identity:

\[
\widetilde{Y}_{(x+dx)x}(t) = t + f(x, t)dx
\]

Which can be sussed out as ``composing an infinitesimal amount,'' or ``composing over the interval $[x,x+dx]$.'' If one can accept this malignant use of infinitesimals\footnote{As the author would argue Classical Analysts took it as fact, though they definitely wrote it differently.}, it can be expanded using our semi-group laws $\widetilde{Y}_{cb}(\widetilde{Y}_{ba}(t)) = \widetilde{Y}_{ca}(t)$, so that:

\begin{eqnarray*}
 \widetilde{Y}_{(x+dx)a} &=& \widetilde{Y}_{(x+dx)x}(\widetilde{Y}_{xa}) = \widetilde{Y}_{xa} + f(x, \widetilde{Y}_{xa}) dx\\
\frac{d \widetilde{Y}_{xa}}{dx} &=& \frac{\widetilde{Y}_{(x+dx)a} - \widetilde{Y}_{xa}}{dx}\\
&=&f(x,\widetilde{Y}_{xa})\\
\end{eqnarray*}

This kind of tells us this idea should work. If the objects converge in the best manner possible, all of this seems like a Leibnizian argument using infinitesimals.\\

The above arguments work out formally as we've written, but proving it does generally and rigorously is difficult. For that reason, we will work through a specific case. It can be illuminating to use an example, and may clear cut some of the block-ways which heed intuition on the matter. This will also give a glimpse of the difficulty of the problem in a rigorous setting.

\section{The nit and gritty}\label{sec5}
\setcounter{equation}{0}

Now that we’re caught up with the sweeping motions, we’ll work through a case in which we can do everything we just did above but with a bit more rigor. To do such, we’ll work with the function $f(x, t) = e^{-xt}$ for $x \in [0, 1]$ and $t \in \mathbb{R}^+$. And we'll try to construct The Compositional Integral of $f$. Although we've just deliberated on The Compositional Integral as a formal thing; the author has yet to construct it, or even prove its existence. We aim to prove the following theorem:

\begin{theorem}\label{thm1B}

The following two claims hold:

\begin{enumerate}
	\item For $0 \le a \le b \le 1$ and $t \in \mathbb{R}^+$, there is a unique Compositional Integral of $e^{-xt}$, denoted
	
	\[
	Y_{ba}(t) = \int_a^b e^{-st}\,ds \bullet t
	\]
	
	Where $Y_{ba}(t): \mathbb{R}^+ \to \mathbb{R}^+$. 
	
	\item Let $P = \{ s_i\}_{i=0}^{n}$ be a partition of $[a, b]$ written in descending order, with $s_{i+1} \le s_i^* \le s_i$; as $\Delta s_i = s_i - s_{i+1} \to 0$ the expression 

	\[
	\OmSum_{i=0}^{n-1}t+e^{-s_i^*t}\Delta s_i \bullet t
	\]

converges to The Compositional Integral $Y_{ba}(t)$ of $e^{-xt}$.
\end{enumerate}
\end{theorem}

In our proof, it will be shown in one motion that $\widetilde{Y}=Y$. It will then be evident the function $\widetilde{Y}_{xa}(t)$ is the unique function $y(x)$ such that $y(a) = t$ and $y'(x) = e^{-xy(x)}$.

This provides us with the quick and justifiable statement that The Compositional Integral is a meaningful thing and looks something like a Riemann Sum, if only a Riemann Sum involved compositions... A Riemann composition, if you will. For convenience, the author will call it \emph{The Riemann Composition} of The Compositional Integral.

The proof we will provide will require some hand waving, as to write out all the steps produces a mess of equations. For this reason we will try to be short but convincing nonetheless. We will try to argue classically, but will admit much more rigor than a Classical Analyst would.

\begin{proof}
To begin, we'll prove ($1$). For all $t\in \mathbb{R}^+$ and $0 \le a \le b \le 1$ there is a function $Y_{ba}(t):\mathbb{R}^+\to\mathbb{R}^+$--The Compositional Integral of $f(x,t) = e^{-xt}$. 

To show this is an exercise in soft-analysis. For all $t_0,t_1 \in\mathbb{R}^+$ and $x \in [0,1]$ we have $|e^{-xt_0} - e^{-xt_1}| \le |t_0 - t_1|$. Therefore by The Picard-Lindel\"{o}f Theorem, for every $x_0 \in [0,1]$ there is a neighborhood $|x-x_0| < \delta$ ($\delta$ can be chosen to work for all $x_0$), where for each $ t \in \mathbb{R}^+$, we have a function $y_{t,x_0}$ in which $y_{t,x_0}'(x)= e^{-xy_{t,x_0}(x)}$ and $y_{t,x_0}(x_0) = t$. These neighborhoods $|x-x_0|<\delta$, and hence functions $y_{t,x_0}$, can be glued together. We can extend $y_{t,x_0}(x)$ from $|x-x_0| < \delta$ to $|x-x_0| < 3\delta/2$ by noticing 

\[
y_{t,x_0}(x\pm\delta/2) = y_{y_{t,x_0}(x_0\pm\delta/2),x_0}(x)
\]

Continuing this process, $y_{t,x_0}$ can be extended from $|x-x_0|<\delta$ to $x \in [0,1]$ using a monodromy principle. The presiding identity principle is not that $y$ is analytic, though. Instead $y_{t,x_0}$ satisfies the same First Order Differential Equation for each $x_0$.

To elaborate: consider two intervals $I$ and $J$ where $I \cap J \neq \emptyset$. Let $u: I \to \mathbb{R}$ and $w: J \to \mathbb{R}$. Assume $u\Big{|}_{I \cap J} = w\Big{|}_{I \cap J}$, and they satisfy the same First Order Differential Equation, $y' = e^{-xy(x)}$. By the uniqueness property of First Order Differential Equations, $u = w$ on $I \cup J$.  This monodromy principle forms $Y_{ba}(t) = y_{t,a}(b)$ for all $t\in \mathbb{R}^+$ and $0 \le a \le b \le 1$. 

Lastly, $Y_{ba}(t) : \mathbb{R}^+ \to \mathbb{R}^+$ because $Y_{aa}(t) = t \in \mathbb{R}^+$ and $Y_{xa}(t)$ is increasing in $x$ because it's derivative is greater than $0$. Theorem \ref{thm1} can now be thought of rigorously, and shows that $Y$ satisfies a group structure.\\

For our proof of ($2$), that The Riemann Composition converges to $Y_{ba}(t)$; by Taylor's theorem:

\[
Y_{ss'}(t) = t + f(s^*,t)(s-s') + R_\Delta
\]

Where here $\frac{R_\Delta}{\Delta} \to 0$ as $\Delta \to 0$, and $\Delta$ is an upper bound on $(s - s')$ where $0 \le s' \le s^* \le s \le 1$. Now $R_\Delta$ depends on $s^*, s', s$ and $t$, but we are going to throw its dependence away, as it can clutter the proof. Since we will be letting $\Delta\to 0$ its dependence on $t$ (and $s^*, s', s$) becomes irrelevant (especially because the convergence is uniform for $t \in \mathbb{R}^+$, and $0 \le s' \le s^* \le s \le 1$). Let $P= \{s_i\}_{i=0}^{n}$ be a partition of $[a,b]$ in descending order, and let $s_{i+1} \le s_i^* \le s_i$. Let $\max_{i=0,1,...,n-1} \Delta s_i=\Delta$. The following identities should illustrate the method of the proof:

\begin{eqnarray*}
Y_{ba}(t) &=& Y_{bs_1}(Y_{s_1s_2}(...Y_{s_{n-1}a}(t)))\\
&=&\OmSum_{i=0}^{n-1} Y_{s_is_{i+1}}(t)\bullet t\\
&=&\OmSum_{i=0}^{n-1} t+ f(s_i^*,t)\Delta s_i + R^{i}_\Delta \bullet t\\
&=& \OmSum_{i=0}^{n-1} t+ f(s_i^*,t)\Delta s_i\bullet t + \sum_{i=0}^{n-1} Q^i_\Delta\\
\end{eqnarray*}

Where here each $\frac{Q^i_\Delta}{\Delta} \to 0$ as $\Delta\to 0$ for all $0 \le i \le n-1$. Ignoring $R_\Delta^i$'s dependence on $t$, the  justification of this identity follows from an inductive use of the rule $g(t+\mathcal{O}(\Delta^2)) = g(t) + \mathcal{O}(\Delta^2)$--where we must be sure to count how many error terms we are adding together. This crude formalism is justified, again due to the uniform convergence of $R_\Delta^i \to 0$.

Now since $\Delta = \mathcal{O}(1/n)$ we must have $Q^i_\Delta = \mathcal{O}(1/n^2)$. We are taking the sum $\sum_{i=0}^{n-1} Q^i_\Delta$, so we can see that 

\[
\sum_{i=0}^{n-1} Q^i_\Delta = \sum_{i=0}^{n-1} \mathcal{O}(1/n^2) =  n \mathcal{O}(1/n^2) = \mathcal{O}(1/n)
\]

This allows us to write that:

\[
Y_{ba}(t) - \OmSum_{i=0}^{n-1} t + f(s_i^*,t)\Delta s_i \bullet t = \mathcal{O}(1/n) = \mathcal{O}(\Delta)
\]

And so in letting $\Delta\to 0$ (and $n\to\infty$), the LHS tends to zero and our Riemann Composition converges to The Compositional Integral of $e^{-xt}$.
\end{proof}

To summarize what was especially needed from $f$ in this argument, in our exact choice of $f(x,t) = e^{-xt}$; the mapping $f(x,t) : [0,1] \times \mathbb{R}^+ \to \mathbb{R}^+$ and therefore the nested compositions are meaningful. Secondly, it was required that the function $f(x,t) = e^{-xt}$ is globally Lipschitz continuous in $t$ on $\mathbb{R}^+$ for all $x \in [0,1]$ as this allowed for the simple argument proving the function $Y_{ba}(t)$ even exists and is unique. The global Lipschitz condition also ensured the uniform convergence of the error term $R^i_\Delta \to 0$, which allowed for the error term to be pulled through the composition so easily.

Supposing we chose another function $f$ where $t$ was restricted to some interval $[c,d]$, then this causes innumerable problems. We would need that compositions of $t+f(s_i^*,t)\Delta s_i$ are meaningful things, but this is difficult as $[c,d]$ is bounded and the composing functions may grow to a value greater than $d$, or less than $c$, and our composition may no longer make sense. Especially because of the dangling translation by $t$. We would need $t+f(s_i^*,t)\Delta s_i:[c,d] \to [c,d]$  for all $0 \le i \le n-1$ and as $\Delta s_i \to 0$, which is unnatural and quite restrictive. Avoiding this would require more clever topological arguments; they would surely not fit in the confines of this notice.

Therein, our choice of $e^{-xt}$ was very intentional, and a very special function for this argument to work. Constructing The Compositional Integral for arbitrary functions proves to be a much more difficult task, especially if the only condition we demand is that $f$ is Lipschitz. But if one takes Euler's word for it, it isn't much of a leap.

The author imagines it very plausible that The Riemann Composition converges to The Compositional Integral if all that is asked is that $f$ is Lipschitz. A proof of this would simply take longer than the space we have in this paper. And probably more expertise than the author has.

\section{In Conclusion}\label{sec6}
\setcounter{equation}{0} 

The Compositional Integral can be made into a meaningful thing. It is a stark redefinition of the Riemann Integral, and provides a productive form of The Picard-Lindel\"{o}f Theorem, in which we have some \emph{thing} and this \emph{thing} converges to the solution of a First Order Differential Equation. It also looks like the integral in more ways than one. The author has remained as curt as possible, but hopes to excise a curiousity in the reader and leave the subject open ended. What else can be done with this strange new integral? Can we add measure theory by looking at $\mu (\Delta s_i)$ rather than $\Delta s_i$ for some measure $\mu$? Can we add contour integrals by parameterizing contours $C$ in the complex plane using some differentiable arc $\gamma$? How do we take limits at infinity? Are there dominated or monotone convergence theorems? The author can only imagine.\\

And as to what we've really done in this paper, it may be fun to hint at expansions of common functions using these methods. We can express $e^x$ in a somewhat new way, or at least provide a new way of justifying the expansion. For $x,t \in \mathbb{R}^+$:

\[
t e^x= \int_{0}^x t\,ds\bullet t
\]

because $y = te^x$ satisfies $y(0) = t$ and $y'(x) = y(x)$. Interestingly, now the group structure of The Compositional Integral become the multiplicative property of $e^x$. In this special case, The Riemann Composition reduces to an identity exactly of the form $\lim_{n\to\infty}(1+\frac{x}{n})^n = e^x$--the author thinks it's one of the many ways Euler probably derived the expression. Namely if $P=\{s_i\}_{i=0}^n$ is a partition of $[0,x]$, then:

\begin{eqnarray*}
te^x &=& \lim_{\Delta s_i \to 0} \OmSum_{i=0}^{n-1} t + t\Delta s_i \bullet t\\
&=& \lim_{\Delta s_i \to 0} \OmSum_{i=0}^{n-1} t(1 + \Delta s_i) \bullet t\\
&=& \lim_{\Delta s_i \to 0} t \prod_{i=0}^{n-1}(1+\Delta s_i)\\
\end{eqnarray*}

Where here $\Delta s_i$ looks like $\frac{x}{n}$; so, $\prod_{i=0}^{n-1}(1+\Delta s_i)$ looks like $(1+\frac{x}{n})^n$. Using the same reasoning, we can generalize. The following identities written as though they are Riemann Sums, are derived in the same manner and are interesting--but are not unknown:

\begin{eqnarray*}
e^{x^2} &=& \lim_{\Delta s_i \to 0} \prod_{i=0}^{n-1} (1 + 2 s_i^*\Delta s_i) = \lim_{n\to\infty} \prod_{i=0}^{n-1} \Big{(}1 + 2 i \Big{(}\frac{x}{n}\Big{)}^2\Big{)}\\
e^{x^3} &=& \lim_{\Delta s_i \to 0} \prod_{i=0}^{n-1} (1 + 3 (s_i^*)^2\Delta s_i) = \lim_{n\to\infty} \prod_{i=0}^{n-1} \Big{(}1 + 3 i^2 \Big{(}\frac{x}{n}\Big{)}^3\Big{)}\\
&\vdots&\\
e^{x^k} &=& \lim_{\Delta s_i \to 0} \prod_{i=0}^{n-1} (1 + k (s_i^*)^{k-1}\Delta s_i) = \lim_{n\to\infty} \prod_{i=0}^{n-1} \Big{(}1 + k i^{k-1} \Big{(}\frac{x}{n}\Big{)}^{k}\Big{)}\\
\end{eqnarray*}

The following derivation is left to the reader:

\[
te^{\int_0^x p(s)\,ds} = \int_0^x p(s)t\,ds\bullet t = \lim_{\Delta s_i \to 0} t\prod_{i=0}^{n-1} (1 + p(s_i^*)\Delta s_i)
\]

Therefore The Compositional Integral of $f$ reduces to the Volterra integral of $p$ when $f(s,t) = p(s)t$ \cite{Volterra}. If I haven't convinced the reader--these identities can be proven by taking logarithms, and using the estimate $\log(1+x) \sim x$, which is the driving point of Volterra's construction.

\section*{Acknowledgement}
The author would like to thank the retired professor Dr. John Gill of Colorado state University-Pueblo. John may be better known as a revolutionary free climber, but the few e-mails and notes we passed back and forth helped me clarify exactly what I meant when I'd say to myself: ``A Riemann Sum, but instead it's like... composition.'' John essentially invented The Riemann Composition of The Picard Lindel\"{o}f Integral--and truly claims priority (though only if Euler doesn't). John chose a different nomenclature, and called it ``The Virtual Integral.'' The author came to the idea independently, but in the altered format presented in this article.  This exposition intended to: expand on John Gill's ideas; give context in the form of Euler's method; define and describe the group structure; and propose more fluid notation. For John Gill's original papers on ``The Virtual Integral,'' the author points to \cite{VirtInt1, VirtInt2}; but the abundance of his notes are sadly unpublished. The majority of his work was made known to the author through personal correspondence. Even still, a lot of his work on ResearchGate frequently makes brief mention of ``The Virtual Integral.'' The author is greatly indebted to his mathematical contributions, and his encouraging comments.

\vfill\eject
\end{document}